\documentclass[12pt]{amsart}
\usepackage{amsmath,amssymb,latexsym,soul,cite,mathrsfs}

\usepackage{color,enumitem,graphicx}
\usepackage[colorlinks=true,urlcolor=blue,
citecolor=red,linkcolor=blue,linktocpage,pdfpagelabels,
bookmarksnumbered,bookmarksopen]{hyperref}
\usepackage[english]{babel}

\usepackage[left=2.9cm,right=2.9cm,top=2.8cm,bottom=2.8cm]{geometry}
\usepackage[hyperpageref]{backref}

\usepackage[colorinlistoftodos]{todonotes}
\makeatletter
\providecommand\@dotsep{5}
\def\listtodoname{List of Todos}
\def\listoftodos{\@starttoc{tdo}\listtodoname}
\makeatother

\numberwithin{equation}{section}

\newtheorem{theorem}{Theorem}[section]

\newtheorem{lemma}[theorem]{Lemma}

\title[Strauss' and Lions' type results in $BV(\mathbb{R}^N)$]
{Strauss' and Lions' type results in $BV(\mathbb{R}^N)$ with an application to 1-Laplacian problem}

\author[G. M. Figueiredo]{Giovany M. Figueiredo}
\author[M. Pimenta]{Marcos T. O. Pimenta}

\address[G. M. Figueiredo]{\newline\indent Faculdade de Matem\'atica
\newline\indent 
Universidade Federal do Par\'a
\newline\indent
66075-110, Bel\'em - PA, Brazil}
\email{\href{mailto:giovany@ufpa.br}{giovany@ufpa.br}}
\address[M. T. O. Pimenta]{\newline\indent Departamento de Matem\'atica e Computa\c{c}\~ao
\newline\indent 
Faculdade de Ci\^encias e Tecnologia
\newline\indent
UNESP - Universidade Estadual Paulista
\newline\indent
19060-900, Presidente Prudente - SP - Brazil}
\email{\href{mailto:pimenta@fct.unesp.br}{pimenta@fct.unesp.br}}

\thanks{Giovany M. Figueiredo was partially
supported by  FAPESP and CNPq, Brazil. Marcos T.O. Pimenta was supported by Fapesp and CNPq, Brazil. }

\begin{document}

\maketitle
\begin{abstract}

In this work we state and prove versions of some classical results, in the framework of functionals defined in the space of functions of bounded variation in $\mathbb{R}^N$. More precisely, we present versions of the Radial Lemma of Strauss, the compactness of the embeddings of the space of radially symmetric functions of $BV(\mathbb{R}^N)$ in some Lebesgue spaces and also a version of the Lions Lemma, proved in his celebrated paper of 1984. As an application, we state and prove a version of the Mountain Pass Theorem without the Palais-Smale condition in order to get existence of a ground-state bounded variation solution of a quasilinear elliptic problem involving the $1-$Laplacian operator in $\mathbb{R}^N$. This seems to be the very first work dealing with stationary problems involving this operator in the whole space.

\end{abstract}
\maketitle

\section{Introduction and some abstract results}

\hspace{.5cm} 

When dealing with semilinear elliptic equations in $\mathbb{R}^N$, the lack of compactness is a problem to be considered. In general, what people are used to do is to impose some symmetry on the problem in order to recover the compactness of the embeddings of the Sobolev space into Lebesgue spaces. In this procedure at least two results are absolutely essential: a kind of Strauss Radial Lemma and a version of the Symmetric Criticality Principle of Palais.

Another very useful tool, mainly when symmetry is broken, is the very known Lions' Lemma, which has been settled by Lions in the celebrated paper \cite{Lions} and widely used since them.

As regards quasilinear problems, depending on some features of the differential operator to be considered, it can be necessary to deal with it in the space of functions of bounded variation, $BV(\mathbb{R}^N)$. This is the case when dealing with the mean-curvature operator or with the $1-$Laplacian operator, a highly singular version of the usual $p-$Laplacian operator with $p = 1$. However, the space $BV(\mathbb{R}^N)$, which is going to be precisely defined later on, hasn't a crucial property that the most part of the Sobolev spaces has, the reflexivity. Indeed, the dual of $BV(\mathbb{R}^N)$ is not well known yet. This lack of reflexivity becomes a very difficult task to find critical points of functionals defined in this space and, as a consequence, we can see few or even no work dealing with elliptic problems in $\mathbb{R}^N$ which are normally modeled in this space. In fact, in general reflexivity is used since the weak limits of sequences, which can be minimizing, Palais-Smale, and so on, are the candidates to be weak solutions of the problems.

In this work, to study a quasilinear problem involving the $1-$Laplacian operator in $\mathbb{R}^N$, in order to deal at once with both, the lack of compactness of the embeddings of $BV(\mathbb{R}^N)$ into Lebesgue spaces and with the lack of reflexivity, we state and prove versions in $BV(\mathbb{R}^N)$ of some classical results, like Lions' Lemma and the compactness of embeddings of a subspace of $BV(\mathbb{R}^N)$ into $L^q(\mathbb{R}^N)$. To be more precise, we prove that $BV_{rad}(\mathbb{R}^N)$, the space of functions in $BV(\mathbb{R}^N)$ which are radially symmetric, is compactly embedded into $L^q(\mathbb{R}^N)$, for $\displaystyle 1 < q < 1^*$, where $\displaystyle 1^* = \frac{N}{N-1}$. Such results are stated beneath:

\begin{theorem}[Lions' Lemma in $BV(\mathbb{R}^N)$]
Suppose there exist $R > 0$, $1 \leq q < 1^*$ and a bounded sequence $(u_n)$ in $BV(\mathbb{R}^N)$ such that
$$
\sup_{y \in \mathbb{R}^N}\int_{B_R(y)}|u_n|^qdx \to 0, \quad \mbox{as $n \to \infty$.}
$$
Then $u_n \to 0$ in $L^s(\mathbb{R}^N)$ for all $s \in (1,1^*)$.
\label{lionslemma}
\end{theorem}

\begin{theorem}
Let  $BV_{rad}(\mathbb{R}^N) = \{u \in BV(\mathbb{R}^N); \, u(x) = u(|x|)\}$. Then the embedding below is compact
$$BV_{rad}(\mathbb{R}^N) \hookrightarrow L^q(\mathbb{R}^N), \quad \mbox{for $1 < q < 1^*$.}$$
\label{theorem1}
\end{theorem}

In the proof of the latter, it is necessary to state and proof a version of the Strauss Radial Lemma (see \cite{Strauss}), which we think that can have interest in itself and is going to be proved later on.

The curious fact about the usage of these results is that, in contrast with the versions of them used in semilinear problems, where in general they are used separately, here we have to use both together in order to find a nontrivial critical point of the functional we are going to analyse. This happen since in a non-reflexive setting, compactness of the embeddings does not imply in Palais-Smale condition. In fact, if $(u_n)$ is a Palais-Smale sequence for the Euler-Lagrange functional, the compactness of the embeddings imply that $\|u_n\| \to \|u\|$, where $u$ is the limit in some sense. However, the lack of reflexivity does not allow to conclude that this imply in $\|u_n - u\| \to 0$. This is why we have to use both of the results together.

As an application of our compactness result, we study the following quasilinear problem

\begin{equation}
\left\{
\begin{array}{rr}
\displaystyle - \Delta_1 u + \frac{u}{|u|} & = f(u) \quad \mbox{in $\mathbb{R}^N$,}\\
& u \in BV(\mathbb{R}^N),
\end{array} \right.
\label{Pintro}
\end{equation}
where the $1-$Laplacian operator is defined by $\displaystyle \Delta_1 u := \mbox{div}\left(\frac{\nabla u}{|\nabla u|}\right)$ and the nonlinearity $f$ satisfies the following set of assumptions:
\begin{itemize}
\item [$(f_1)$] $f \in C(\mathbb{R})$;
\item [$(f_2)$] $f(s) = o(1)$ as $s \to 0$;
\item [$(f_3)$] There exist constants $c_1, c_2 > 0$ and $p \in (1,1^*)$ such that
$$|f(s)| \leq c_1 + c_2|s|^{p-1};$$
\item [$(f_4)$] There exists $\theta > 1 $ such that $$0 < \theta F(s) \leq f(s)s, \quad \mbox{for $s \neq 0$},$$
where $F(s) = \int_0^s f(t)dt$;
\item [$(f_5)$] $f$ is increasing.
\end{itemize}

In fact we prove the following result, which states the existence of a ground-state nontrivial solution of (\ref{Pintro}).

\begin{theorem}
Suppose that $f$ satisfies  the conditions $(f_1) - (f_5)$. Then there exists a ground-state solution $u \in BV(\mathbb{R}^N)$ of (\ref{Pintro}).
\label{theoremapplication}
\end{theorem}

Our approach to prove Theorem \ref{theoremapplication} is variational and, in spite of the most part of works dealing with the 1-Laplacian operator,  we work in $BV(\mathbb{R}^N)$ itself, rather than extend the energy functional to some Lebesgue space. With this in mind, we have to overcome the lack of the Palais-Smale condition, which in fact is a not consequence of some lack of compactness of embeddings of $BV(\mathbb{R}^N)$ (since we overcome this by working with radial functions), but it comes from the weak proprieties of convergence that the space of bounded variation functions has. 

Because of these difficulties we prove Theorem \ref{theoremapplication} by using a version of the Mountain Pass Theorem to locally Lipschitz functionals, in the absense of the Palais-Smale condition, that we state and prove here, since we could not find it in the literature.

\begin{theorem}
Let $E$ be a Banach space, $\Phi = I_0 - I$ where $I \in C^1(E,\mathbb{R})$ and $I_0$ is a locally Lipschitz convex functional defined in $E$.
Suppose that the functional $\Phi$ satisfies:
\begin{itemize}
\item [$i)$] There exist $\rho > 0$, $\alpha > \Phi(0)$ such that $\displaystyle \Phi|_{\partial B_\rho(0)} \geq \alpha$,
\item [$ii)$] $\Phi(e) < \Phi(0)$ for some $e \in E \backslash \overline{B_\rho(0)}$.
\end{itemize}
Then for all $\epsilon > 0$ there exists $x_\epsilon \in E$ such that 
\begin{equation}
c - \epsilon < \Phi(x_\epsilon) < c+\epsilon,
\label{MP1}
\end{equation}
where $c \geq \alpha$ is characterized by
\begin{equation}
c = \inf_{\gamma \in \Gamma} \sup_{t \in [0,1]}\Phi(\gamma(t)),
\label{minimax}
\end{equation}
where $\Gamma = \{\gamma \in C^0([0,1],E); \, \gamma(0) = 0 \, \,  \mbox{and} \, \,  \gamma(1) = e\}$
and
\begin{equation}
I_0(y) - I_0(x_\epsilon) \geq I'(x_\epsilon)(y - x_\epsilon) - \epsilon \|y-x_\epsilon\|, \quad \forall y \in E.
\label{MP2}
\end{equation}
\label{mountainpass}
\end{theorem}

In the last years an increasing number of researchers have dedicated their efforts studying problems involving the $1-$Laplacian operator. A version of Br\'ezis-Nirenberg problem to $1-$Laplacian has been studied in \cite{DegiovanniMagrone} by Degiovanni and Magrone, where they use a nonstandard linking structure in order to get solutions of the problem. In \cite{LeonWebler}, Le\'on and Webler study a parabolic problem involving the $1-$Laplacian operator and succeed in proving global existence and uniqueness for source and initial data in some adequate space. In \cite{FigueiredoPimenta}, the authors seems to be the pioneers in using Nehari types arguments in order to get bounded variation solutions for problems involving the mean-curvature or the $1-$Laplacian operators.

In what is concerned with the approach one can follow in studying $1-$Laplacian problems in bounded domains, roughly speaking there are two ways that can be considered. One can study $\Delta_1$ through $p-$Laplacian problems and then taking the limit as $p \to 1^+$, like in \cite{Demengel,Demengel1}, or one can directly deal with $\Delta_1$ itself, by using variational methods for instance. However, to precisely understand (\ref{Pintro}), one have to replace the expression $\displaystyle \frac{\nabla u}{|\nabla u|}$ by a well defined vector field which extend the former wherever $\nabla u$ vanishes and similarly, one has to substitute $\displaystyle\frac{u}{|u|}$ by a set-valued function to give meaning to this expression wherever $u$ vanishes. This kind of procedure can be seen in details in \cite{Kawohl} and also in \cite{DegiovanniMagrone}.

By using a variational approach, we have to deal with an Euler-Lagrange functional which is not smooth, although locally Lipschitz. Hence the way in which the functional and its Euler-Lagrange equation is linked is somehow tricky. In fact the sense of solution we consider here has to take into account the concept of generalized gradient developed by Clarke (see \cite{Clarke,Chang}). More precisely, the Euler-Lagrange functional of (\ref{Pintro}) is modeled in a subspace of $BV(\mathbb{R}^N)$ and is given by
$$
\Phi(u) = \int_{\mathbb{R}^N}|Du| + \int_{\mathbb{R}^N}|u|dx - \int_{\mathbb{R}^N}F(u)dx,
$$
where $Du$ is the distributional derivative of $u$, which in turn is a Radon measure. As can be seen in Section 3, we say that $u$ is a bounded variation solution of (\ref{Pintro}) if
$$
\mathcal{J}(v) - \mathcal{J}(u) \geq \int_{\mathbb{R}^N}f(u)(v-u) dx,
$$
for all $v \in BV(\mathbb{R}^N)$, where
$$
\mathcal{J}(u) = \int_{\mathbb{R}^N}|Du| + \int_{\mathbb{R}^N}|u|dx.
$$

In the end of the proof, in order to assure that the critical point of the restricted functional in fact is a critical point of $\Phi$ in all of $BV(\mathbb{R}^N)$, we need a version of the Symmetric Criticality Principle of Palais for non-smooth functionals defined in possibly non-reflexive Banach spaces. This is provided by Squassina in \cite{Squassina} and is used in the final of Section 4.

The paper is organized as follows. In Section 2 we perform some preliminary explanation about the space $BV(\mathbb{R}^N)$. In Section 3 we prove the abstract results we are going to use in the next sections. In Section 4 we present an application of the abstract results in order to get ground state solutions to a $1-$Laplacian problem.

The authors would like to warmly thank Prof. Alexandru Krist\'aly for some discussions about the Symmetric Criticality Principle of Palais. This work was written while Giovany M. Figueiredo was as a Visiting Professor at FCT - Unesp in Presidente Prudente - SP. He would like to thanks the warm hospitality.

\section{Preliminaries}

First of all let us introduce the space of functions of bounded variation, $BV(\mathbb{R}^N)$. We say that $u \in BV(\mathbb{R}^N)$, or is a function of bounded variation, if $u \in L^1(\mathbb{R}^N)$, and its distributional derivative $Du$ is a vectorial Radon measure, i.e., 
$$BV(\mathbb{R}^N) = \left\{u \in L^1(\mathbb{R}^N); \, Du \in \mathcal{M}(\mathbb{R}^N,\mathbb{R}^N)\right\}.$$
It can be proved that $u \in BV(\mathbb{R}^N)$ is equivalent to $u \in L^1(\mathbb{R}^N)$ and
$$\int_{\mathbb{R}^N} |Du| := \sup\left\{\int_{\mathbb{R}^N} u \mbox{div}\phi dx; \, \, \phi \in C^1_c(\mathbb{R}^N,\mathbb{R}^N), \, \mbox{s.t.} \, \, |\phi|_\infty \leq 1\right\} < +\infty.$$

The space $BV(\mathbb{R}^N)$ is a Banach space when endowed with the norm
$$\|u\| := \int_{\mathbb{R}^N} |Du| + |u|_1,$$
which is continuously embedded into $L^r(\mathbb{R}^N)$ for all $\displaystyle r \in \left[1,1^*\right]$.

As one can see in \cite{Buttazzo}, the space $BV(\mathbb{R}^N)$ has different convergence and density properties than the usual Sobolev spaces. For example, $C^\infty_0(\mathbb{R}^N)$ is not dense in $BV(\mathbb{R}^N)$ with respect to the strong convergence, since the closure of $C^\infty_0(\mathbb{R}^N)$ in the norm of $BV(\mathbb{R}^N)$ is equal to $W^{1,1}(\mathbb{R}^N)$, which is a proper subspace of $BV(\mathbb{R}^N)$. This has motivated people to define a weaker sense of convergence in $BV(\mathbb{R}^N)$, called {\it intermediate convergence}. We say that $(u_n) \subset BV(\mathbb{R}^N)$ converge to $u \in BV(\mathbb{R}^N)$ in the sense of the intermediate convergence if 
$$
u_n \to u, \quad \mbox{in $L^1(\mathbb{R}^N)$}
$$
and
$$
\int_{\mathbb{R}^N}|Du_n| \to \int_{\mathbb{R}^N}|Du|,
$$
as $n \to \infty$. Fortunately, with respect to the intermediate convergente, $C^\infty_0(\mathbb{R}^N)$ is dense in $BV(\mathbb{R}^N)$. This fact is going to be used later.

For a vectorial Radon measure $\mu \in \mathcal{M}(\mathbb{R}^N,\mathbb{R}^N)$, we denote by $\mu = \mu^a + \mu^s$ the usual decomposition stated in the Radon Nikodyn Theorem, where $\mu^a$ and $\mu^s$ are, respectively, the absolute continuous and the singular parts with respect to the $N-$dimensional Lebesgue measure $\mathcal{L}^N$. We denote by $|\mu|$, the absolute value of $\mu$, the scalar Radon measure defined like in \cite{Buttazzo}[pg. 125]. By $\displaystyle \frac{\mu}{|\mu|}(x)$ we denote the usual Lebesgue derivative of $\mu$ with respect to $|\mu|$, given by
$$\frac{\mu}{|\mu|}(x) = \lim_{r \to 0}\frac{\mu(B_r(x))}{|\mu|(B_r(x))}.$$ 

It can be proved that $\mathcal{J}: BV(\mathbb{R}^N) \to \mathbb{R}$, given by
\begin{equation}
\mathcal{J}(u) = \int_{\mathbb{R}^N} |Du| + \int_{\mathbb{R}^N} |u|dx,
\label{J}
\end{equation}
is a convex functional and Lipschitz continuous in its domain. It is also well know that $\mathcal{J}$ is lower semicontinuous with respect to the $L^r(\mathbb{R}^N)$ topology, for $r \in [1,1^*]$ (see \cite{Giusti} for example). Although non-smooth, the functional $\mathcal{J}$ admits some directional derivatives. More specifically, as is shown in \cite{Anzellotti}, given $u \in BV(\mathbb{R}^N)$, for all $v \in BV(\mathbb{R}^N)$ such that $(Dv)^s$ is absolutely continuous with respect to $(Du)^s$, it follows that
\begin{equation}
\mathcal{J}'(u)v = \int_{\mathbb{R}^N} \frac{(Du)^a(Dv)^a}{|(Du)^a|}dx + \int_{\mathbb{R}^N} \frac{Du}{|Du|}(x)\frac{Dv}{|Dv|}(x)|(Dv)|^s + \int_{\mathbb{R}^N}\mbox{sgn}(u) v dx,
\label{Jlinha}
\end{equation}
where $\mbox{sgn}(u(x)) = 0$ if $u(x) = 0$ and $\mbox{sgn}(u(x)) = u(x)/|u(x)|$ if $u(x) \neq 0$.
In particular, note that, for all $u \in BV(\mathbb{R}^N)$,
\begin{equation}
\mathcal{J}'(u)u = \mathcal{J}(u).
\label{derivadaJ}
\end{equation}

We have also that $BV(\mathbb{R}^N)$ is a {\it lattice}, i.e., if $u,v \in BV(\mathbb{R}^N)$, then $\max\{u,v\}, \min\{u,v\} \in BV(\mathbb{R}^N)$ and also
\begin{equation}
\mathcal{J}(\max\{u,v\}) + \mathcal{J}(\min\{u,v\}) \leq \mathcal{J}(u) + \mathcal{J}(v), \quad \forall u,v \in BV(\mathbb{R}^N).
\label{lattice}
\end{equation}

\section{Proof of the abstract results}

In order to prove Theorem \ref{theorem1}, we first state and prove a counterpart in the space of functions of bounded variation, of a very important result of Strauss (see \cite{Strauss}), with so many applications when dealing with radial functions in Sobolev spaces.

\begin{lemma}[Radial Lemma in $BV$]
Let $u \in BV_{rad}(\mathbb{R}^N)$, then for almost every $x \in \mathbb{R}^N\backslash \{0\}$, it follows that
$$|u_n(x)| \leq \frac{1}{|x|^{N-1}}\|u_n\|.$$
\label{radiallemma}
\end{lemma}
\begin{proof}
First of all let us note that by \cite{EvansGariepy}[Section 5.2.2], $C^\infty(\mathbb{R}^N)\cap BV(\mathbb{R}^N)$ is dense in $BV(\mathbb{R}^N)$ with respect to the intermediate topology. Then, for $u \in BV_{rad}(\mathbb{R}^N)$, there exists $(u_n) \subset C^\infty(\mathbb{R}^N)\cap BV_{rad}(\mathbb{R}^N)$ such that
\begin{equation}
u_n \to u \quad \in L^1(\mathbb{R}^N)
\label{Radial1}
\end{equation}
and
\begin{equation}
\int_{\mathbb{R}^N}|\nabla u_n| dx \to \int_{\mathbb{R}^N}|Du|,
\label{Radial2}
\end{equation}
as $n \to \infty$.
Denoting $v(x)=v(|x|)=v(r)$ whenever $v$ is a radial function of $\mathbb{R}^N$, we have that
$$
\frac{d}{d\rho}\left(\rho^{N-1}|u_n(\rho)|\right) = (N-1)\rho^{N-2}|u_n(\rho)| + \rho^{N-1}\frac{u_n}{|u_n|}u_n'(\rho) \quad \forall \rho > 0.
$$
Integrating both sides over $(r, +\infty)$ we have that
\begin{eqnarray*}
\int_r^{+\infty} \frac{d}{d\rho}\left(\rho^{N-1}|u_n(\rho)|\right)d\rho &= & \int_r^{+\infty}(N-1)\rho^{N-2}|u_n(\rho)|d\rho + \int_r^{+\infty}\rho^{N-1}\frac{u_n}{|u_n|}u_n'(\rho)d\rho\\
& \geq & \int_r^{+\infty}\rho^{N-1}\frac{u_n}{|u_n|}u_n'(\rho)d\rho.
\end{eqnarray*}
Note that, since $u_n$ decay at infinity and then
\begin{eqnarray*}
r^{N-1}|u_n(r)| & \leq & \int_r^{+\infty}|u_n'(\rho)|\rho^{N-1}d\rho\\
& \leq & \int_{\mathbb{R}^N} |\nabla u_n|dx.
\end{eqnarray*}
Then we have that
\begin{equation}
|u_n(r)| \leq \frac{1}{r^{N-1}}\int_{\mathbb{R}^N} |\nabla u_n|dx.
\label{Radial3}
\end{equation}
Hence by (\ref{Radial1}), (\ref{Radial2}) and (\ref{Radial3}) it follows that 
$$
|u(r)| \leq \frac{1}{r^{N-1}}\int_{\mathbb{R}^N} |Du|, \quad \mbox{a.e. in $\mathbb{R}^N$.}
$$
\end{proof}

Now let us going to the proof of the compactness result.

\begin{proof}[Proof of Theorem \ref{theorem1}]

Let $(u_n) \subset BV_{rad}(\mathbb{R}^N)$ be a bounded sequence and let $C > 0$ be such that
$$
\|u_n\| \leq C, \quad \forall n \in \mathbb{N}.
$$
By Lemma \ref{radiallemma} it follows that, for all $n \in \mathbb{N}$,
$$
|u_n(x)| \leq \frac{C}{|x|^{N-1}}\, \quad \mbox{a.e. in $\mathbb{R}^N \backslash\{0\}$.}
$$
Since $q > 1$, given $\epsilon > 0$, there exists $R > 0$ such that, for all $n \in \mathbb{N}$,
$$
|u_n(x)|^q \leq \frac{\epsilon}{2C}|u_n(x)| \quad \forall x \in B_R(0)^c.
$$
This implies that
\begin{equation}
\int_{B_R(0)^c}|u_n|^q dx \leq  \frac{\epsilon}{2C} \int_{B_R(0)^c}|u_n|dx \leq \frac{\epsilon}{2C} \|u_n\| \leq \frac{\epsilon}{2},
\label{Radial4}
\end{equation}
for all $n \in \mathbb{N}$.
Moreover, since $BV(B_R(0))$ is compactly embedded into $L^q(B_R(0))$, there exists $u \in L^q(B_R(0))$ such that, up to a subsequence $u_n \to u$ in $L^q(B_R(0))$, as $n \to \infty$. Then there exists $n_0 \in \mathbb{N}$ such that
\begin{equation}
\int_{B_R(0)}|u_n - u|^qdx < \frac{\epsilon}{2}, \quad \forall n \geq n_0.
\label{Radial5}
\end{equation}
Let us define $\overline{u}:\mathbb{R}^N \to \mathbb{R}$ as to be equal to $u$ in $B_R(0)$ and equal to $0$ in $B_R(0)^c$. Then, by (\ref{Radial4}) and (\ref{Radial5}), it follows that
\begin{eqnarray*}
\int_{\mathbb{R}^N}|u_n - \overline{u}|^qdx & = & \int_{B_R(0)}|u_n - \overline{u}|^qdx + \int_{B_R(0)^c}|u_n|^qdx\\
& < & \epsilon.
\end{eqnarray*}
Then it is clear that $u_n \to \overline{u}$ in $L^q(\mathbb{R}^N)$, as $n \to \infty$.
\end{proof}

Now let us prove the version of the Mountain Pass Theorem (Theorem \ref{mountainpass}) we are using here. 
Before to start, let us prove that condition (\ref{MP2}) is equivalent to the existence of $z_\epsilon \in E^*$ such that $\|z_\epsilon\|_* \leq \epsilon$ and
\begin{equation}
I_0(y) - I_0(x_\epsilon) \geq I'(x_\epsilon)(y - x_\epsilon) + \langle z_\epsilon,y-x_\epsilon\rangle_{E^*,E}, \quad \forall y \in E,
\label{MP3}
\end{equation}
where $\langle \cdot,\cdot \rangle_{E^*,E}$ denotes the duality pair between $E$ and its dual.

In fact, clearly (\ref{MP3}) implies (\ref{MP2}).
In order to prove that (\ref{MP2}) also imply (\ref{MP3}), let us state a lemma proved by Szulkin in \cite{Szulkin}[Lemma 1.3].

\begin{lemma}
Let $E$ be a Banach space and $\chi:E \to (-\infty,+\infty]$ a lower semicontinuous convex function with $\chi(0) = 0$. If 
$$\chi(x) \geq -\|x\|, \quad \forall x \in E,$$
then there exists $z \in E^*$, $\|z\|_* \leq 1$, such that
$$\chi(x) \geq \langle z,x \rangle_{E^*,E}, \quad \forall x \in E.$$
\label{lemmaszulkin}
\end{lemma}

Now, if (\ref{MP2}) holds, then
$$\frac{1}{\epsilon}\left(I_0((y-x_\epsilon)+x_\epsilon) - I_0(x_\epsilon) - I'(x_\epsilon)(y - x_\epsilon)\right) \geq -\|y-x_\epsilon\|,$$
for all $y \in E$. By applying Lemma \ref{lemmaszulkin} to
$$\chi(x) = \frac{1}{\epsilon}\left(I_0(x + x_\epsilon) - I_0(x_\epsilon) - I'(x_\epsilon)x\right)$$
it follows that there exists $z \in E^*$, such that $\|z\|_* \leq 1$ and
$$\chi(x) \geq \langle z,x \rangle_* \quad \forall x \in E.$$
Taking $z_\epsilon = \epsilon z$ and $x = y-x_\epsilon$ where $y \in E$, it follows (\ref{MP3}) for $z_\epsilon$ and $\|z_\epsilon\|_* \leq \epsilon$.

To proceed with the proof, we need a version of Deformation Lemma without the Palais-Smale condition which has been proved in \cite{FigueiredoPimenta}[Theorem 4]. By the sake of completeness we state and prove it again here.
\begin{theorem}[Deformation lemma]

Let $E$ be a Banach space and $T:E \to \mathbb{R}$ a locally Lipschitz functional. When $a \in \mathbb{R}$, let us denote $T_a = \{x \in E; \, T(x) \leq a\}$. If there exist $d \in \mathbb{R}$, $S \subset E$ and $\alpha, \delta, \epsilon_0 > 0$ such that 
$$\beta(x):= \min\{\|z\|_{E^*}; \, z \in \partial T(x)\} \geq \alpha, \quad \forall x \in T^{-1}([d-\epsilon_0,d+\epsilon_0]) \cap S_{2\delta},$$
where $S_{2\delta}$ is a $2\delta-$neighborhood of $S$, then for $0 < \epsilon < \min\left\{\frac{\delta\alpha}{2},\epsilon_0\right\}$, there exists an homeomorphism $\eta:E \to E$ such that
\begin{itemize}
\item [$i)$] $\eta(x) = x$ for all $x \not \in  T^{-1}([d-\epsilon_0,d+\epsilon_0])\cap S_{2\delta}$;
\item [$ii)$] $\eta(T_{d+\epsilon}\cap S) \subset T_{d-\epsilon}$;
\item [$iii)$] $T(\eta(x)) \leq T(x)$, for all $x \in E$.
\end{itemize}
\label{deformationlemma}
\end{theorem}

\begin{proof}[Proof of Theorem \ref{deformationlemma}]
To start with, under these assumptions, let us recall Lemma 3.3 of \cite{Chang}, which states the existence of a {\it psudo-gradient vector field} for $T$, given by a locally Lipschitz vector field $g: T^{-1}([d-\epsilon_0,d+\epsilon_0])\cap S_{2\delta} \to E$ satisfying
\begin{equation}
\|g(x)\| < 1
\label{vectorfield1}
\end{equation}
and 
\begin{equation}
\langle z^*,g(x)\rangle_{E^*,E} > \frac{\alpha}{2}, \quad \forall z^* \in \partial T(x).
\label{vectorfield2}
\end{equation}

For
\begin{equation}
0 < \epsilon < \min\left\{\frac{\delta\alpha}{2},\epsilon_0\right\},
\label{epsilon}
\end{equation}
define
$$A = T^{-1}([d-\epsilon_0,d+\epsilon_0]) \cap S_{2\delta},$$
$$B = T^{-1}([d-\epsilon,d+\epsilon]) \cap S_\delta$$
and note that $B \subset A$. Define 
$$\psi(x) = \frac{d(x,E\backslash A)}{d(x,E\backslash A) + d(x,B)}$$
and note that $\psi$ is a locally Lipschitz continuous function such that $0 \leq \psi \leq 1$ and
$$
\psi(x) = \left\{
\begin{array}{ll}
1 & \mbox{if $x \in B$,}\\
0 & \mbox{if $x \in E\backslash A$.}\\
\end{array} \right.
$$
Now consider $V(x) = \psi(x)g(x)$ which is also locally Lipschitz continuous and $\sigma(t,x)$ the solution of
$$
\left\{
\begin{array}{rl}
\displaystyle \frac{d}{dt}\sigma(t,x) & = -V(\sigma(t,x)), \quad t > 0,\\
\sigma(0,x) & = x,
\end{array}\right.
$$
which is continuous in $\mathbb{R}_+ \times E$.

Let us choose
\begin{equation}
t_0 \in \left(\frac{2\epsilon}{\alpha}, \delta\right)
\label{t_0}
\end{equation}
and define 
$$
\eta(x) = \sigma(t_0,x), \quad x \in E.
$$

Note that since $V \equiv 0$ in $E \backslash (T^{-1}([d-\epsilon_0,d+\epsilon_0]) \cap S_{2\delta}$, it follows that $i)$ holds.

To prove $ii)$, let us first recall Proposition 9 in \cite{Chang} which implies that $t \mapsto T(\sigma(t,x))$ is a.e. differentiable, for each $x \in E$. Moreover, we have that
\begin{equation}
\begin{array}{lll}
\displaystyle \frac{d}{dt}T(\sigma(t,x)) & \leq & \max\left\{\left< z^*,\frac{d}{dt}\sigma(t,x)\right>_{E^*,E}; \, z^* \in \partial T(\sigma(x,t))\right\}\\
& = & \displaystyle  -\min\{\left< z^*,V(\sigma(t,x))\right>; \, z^* \in \partial T(\sigma(x,t))\}\\
& \leq & \displaystyle  \left\{
\begin{array}{rl}
\displaystyle  - \frac{\alpha}{2} & \mbox{if $\sigma(t,x)\subset T^{-1}([d-\epsilon,d+\epsilon])\cap S_\delta$}\\
0 & \mbox{otherwise,}
\end{array}\right.
\end{array}
\label{estimativadeformacao}
\end{equation}
where we use (\ref{vectorfield2}) in the last estimate. Then the function $t \mapsto T(\sigma(t,x))$ is nonincreasing, for all $x \in E$ and then we get $iii)$.

Note also that, for all $t > 0$
\begin{eqnarray*}
\|\sigma(t,x) - x\| & = & \|\sigma(t,x) - \sigma(0,x)\|\\
& = & \left\| \int_0^t \frac{d}{ds}\sigma(s,x)ds\right\|\\
& \leq & \int_0^t \|V(\sigma(s,x))\|ds\\
& \leq & t.
\end{eqnarray*}

Let us take $x \in T_{d+\epsilon} \cap S$. If there exists some $t \in \left[0,t_0\right]$ such that $T(\sigma(t,x)) < d-\epsilon$, then $T(\sigma(t_0,x)) < d-\epsilon$ and $ii)$ is satisfied by $\eta$. Then suppose that
$$\sigma(t,x) \in T^{-1}([d-\epsilon,d+\epsilon]), \forall t \in [0,t_0]$$
and let us prove that $\sigma(t,x) \subset S_\delta$, $\forall t \in [0,t_0]$. In fact, note that
$$\|\sigma(t,x) - x\| \leq t \leq t_0 < \delta, \quad \forall t \in [0,t_0].$$

Hence, since $\sigma([0,t_0],x) \subset T^{-1}([d-\epsilon,d+\epsilon]) \cap S_\delta$, it follows by (\ref{t_0}) and (\ref{estimativadeformacao}) that

\begin{eqnarray*}
T(\eta(x))& = & T(\sigma(t_0,x))\\
& = & T(x) + \int_0^{t_0} \frac{d}{dt}T(\sigma(s,x))dx\\
& \leq & T(x) - \frac{\alpha}{2}t_0\\
& < & d -\epsilon
\end{eqnarray*}
and $ii)$ follows.
\end{proof}

Now, finally, let us proceed with the proof of Theorem \ref{mountainpass}.
\begin{proof}[Proof of Theorem \ref{mountainpass}]
First, note that since $\Phi(e) < \Phi(0) < \alpha \leq \Phi\left|_{\partial B_\rho}\right.$, then 
$$c \geq \alpha.$$

Suppose by contradiction that there exists $\epsilon > 0$, which can be assumed to satisfy 
$$c - \epsilon > \Phi(0),$$
such that for all $x \in \Phi^{-1}([c-\epsilon,c+\epsilon])$ where $c$ is defined in (\ref{minimax}), (\ref{MP2}) does not hold with $x_\epsilon = x$. Since (\ref{MP2}) is equivalent to (\ref{MP3}), this implies that for all $z \in E^*$ such that $\|z\|_* \leq \epsilon$, there exists $y_\epsilon \in X$, such that
$$
I_0(y_\epsilon) - I_0(x) < I'(x)(y_\epsilon - x)dx + \langle z_\epsilon,y_\epsilon-x\rangle_*.
$$
Hence, it follows that 
$$\beta(x) \geq \epsilon, \quad \forall x \in \Phi^{-1}([c-\epsilon,c+\epsilon]),$$
where $\beta(x) = \inf\{\|w^*\|_*; \, \, w^* \in \partial I_0(u) - I'(x)\}$.

By Theorem \ref{deformationlemma} applied to $T = \Phi$, $d=c$, $\alpha = \epsilon$ and $\epsilon_0 = \epsilon$ it follows that there exists an homeomorphism $\eta: E \to E$ and $\bar{\epsilon} \in (0,\epsilon)$ such that 
\begin{itemize}
\item [$i)$] $\eta(x) = x$ for all $x \not \in  \Phi^{-1}([c-\epsilon,c+\epsilon])$;
\item [$ii)$] $\eta(\Phi_{c+\bar{\epsilon}}) \subset \Phi_{c-\bar{\epsilon}}$.
\end{itemize}

By the definition of $c$, there exists $\gamma \in \Gamma$ such that
$$c \leq \max_{t \in [0,1]}\Phi(\gamma(t)) \leq c+\overline\epsilon.$$

Let us consider $\tilde{\gamma}(t) = \eta(\gamma(t))$ and note that, since $\Phi(0), \Phi(e) < c-\epsilon$, $i)$ implies that $\tilde{\gamma} \in \Gamma$.
Then, $ii)$ implies that
$$c \leq \max_{t\in[0,1]}\Phi(\tilde{\gamma}(t)) \leq c - \overline{\epsilon},$$
which is a contradiction. Then the result follows.

\end{proof}

To end up this section, let us present the proof of the Lions' type result.

\begin{proof}[Proof of Lemma \ref{lionslemma}]
Let $q < s < 1^*$ and $u \in BV(\mathbb{R}^N)$. Since $BV(\mathbb{R}^N) \hookrightarrow L^r(\mathbb{R}^N)$ for $1 \leq r < 1^*$, then $u \in L^q(\mathbb{R}^N)$ and $u \in L^{1^*}(\mathbb{R}^N)$. 

For $R > 0$, by interpolation inequality with $\displaystyle \theta = \frac{s-q}{1^* - q}\frac{1^*}{s}$, it follows that $0 < \theta < 1$ and
\begin{eqnarray*}
|u|_{L^s(B_R(y))} & \leq & |u|_{L^q(B_R(y))}^{1-\theta}|u|_{L^{1^*}(B_R(y))}^\theta\\
& \leq & c|u|_{L^q(B_R(y))}^{1-\theta}\|u\|_{BV(B_R(y))}^\theta.
\end{eqnarray*}
Covering $\mathbb{R}^N$ by balls of radius $R$ and center in $(y_n)$ in such a way that each point in $\mathbb{R}^N$ belongs to at maximum $N+1$ balls, we have that
\begin{eqnarray*}
\int_{\mathbb{R}^N}|u|^sdx & \leq & \sum_{n = 1}^{+\infty}\int_{B_R(y_n)}|u|^sdx\\
& \leq & c\sum_{n = 1}^{+\infty}|u|_{L^q(B_R(y_n))}^{(1-\theta)s}\|u\|_{BV(B_R(y_n))}^{\theta s}\\
& \leq & c\left(\sup_{y \in \mathbb{R}^N}\int_{B_R(y)}|u|^qdx \right)^\frac{(1-\theta)s}{q} \lim_{k \to +\infty}\sum_{n = 1}^k\left(\int_{B_R(y_n)}|Du| + \int_{B_R(y_n)}|u|dx\right)^{\theta s}\\
& = & c\left(\sup_{y \in \mathbb{R}^N}\int_{B_R(y)}|u|^qdx \right)^\frac{(1-\theta)s}{q} \lim_{k \to +\infty}\sum_{n = 1}^k\left(\int_{\mathbb{R}^N}\chi_{B_R(y_n)}|Du| + \int_{\mathbb{R}^N}\chi_{B_R(y_n)}|u|dx\right)^{\theta s}\\
& = & c\left(\sup_{y \in \mathbb{R}^N}\int_{B_R(y)}|u|^qdx \right)^\frac{(1-\theta)s}{q} \lim_{k \to +\infty}\left(\int_{\mathbb{R}^N}\sum_{n = 1}^k\chi_{B_R(y_n)}|Du| + \int_{\mathbb{R}^N}\sum_{n = 1}^k \chi_{B_R(y_n)}|u|dx\right)^{\theta s}\\
& \leq & c\left(\sup_{y \in \mathbb{R}^N}\int_{B_R(y)}|u|^qdx \right)^\frac{(1-\theta)s}{q}(N+1)\|u\|^{\theta s}.
\end{eqnarray*}

Then, since $(u_n)$ is bounded in $BV(\mathbb{R}^N)$, by the last inequality and the hypothesis, it follows that
\begin{equation}
u_n \to 0, \quad \mbox{in $L^s(\mathbb{R}^N)$,}
\label{Lions1}
\end{equation}
for all $ q < s < 1^*$.

Then, if $q = 1$ we are done. Otherwise, if $1 < q < 1^*$, let us consider $1 < s \leq q$ and take $s_0 \in (q,1^*)$ in such a way that (\ref{Lions1}) holds.
Note that $u \in L^1(\mathbb{R}^N)\cap L^{s_0}(\mathbb{R}^N)$ and, since $s \in (1,s_0)$, by doing 
$$\theta = \frac{s_0 - s}{s(s_0 - 1)}$$
we have that
$$
\frac{1}{s} = \frac{\theta}{1} + \frac{1-\theta}{s_0} \quad \mbox{and} \quad 0 < \theta < 1.
$$
Then, again by interpolation inequality, the embedding of $BV(\mathbb{R}^N)$ and (\ref{Lions1}), it follows that
$$
|u_n|_s \leq |u_n|_1|u_n|_{s_0} \leq \|u_n\||u_n|_{s_0} \to 0
$$
as $n \to \infty$, since $(u_n)$ is bounded in $BV(\mathbb{R}^N)$. This completes the proof.
\end{proof}

\section{Application}

In this section we present an application of Theorem \ref{theorem1} to the following problem
\begin{equation}
\left\{
\begin{array}{rl}
\displaystyle - \Delta_1 u + \frac{u}{|u|} & = f(u) \quad \mbox{in
$\mathbb{R}^N$,}\\
u \in BV(\mathbb{R}^N),
\end{array} \right.
\label{P1}
\end{equation}
where $f$ satisfies the conditions $(f_1) - (f_5)$.

Since (\ref{P1}) is variational, let us define the energy functional associated to it, $\Phi: BV(\mathbb{R}^N) \to \mathbb{R}$ given by
\begin{equation}
\Phi(u) = \mathcal{J}(u) - \mathcal{F}(u),
\label{Phi}
\end{equation}
where $\mathcal{F}: BV(\mathbb{R}^N) \to \mathbb{R}$ is defined by
\begin{equation}
\mathcal{F}(u) = \int_{\mathbb{R}^N}F(u)dx
\label{F}
\end{equation}
and $\mathcal{J}$ is given by (\ref{J}).

It is a simple matter to prove that $\mathcal{F}$ is a smooth functional. Moreover, by (\ref{Jlinha}), $\mathcal{J}'(u)u = \mathcal{J}(u)$ for all $u \in BV(\mathbb{R}^N)$. Then, the directionals derivatives $\Phi'(u)u$ exists and
\begin{equation}
\Phi'(u)u = \mathcal{J}(u) - \int_{\mathbb{R}^N}f(u)udx.
\label{Phiderivada}
\end{equation}

Before we start to deal with this equation, let us make precise the sense of solution we are considering here. Since $\Phi$ can we written as the difference between a Lipschitz and a smooth functional in $BV(\mathbb{R}^N)$, we say that $u_0 \in BV(\mathbb{R}^N)$ is a solution of (\ref{P1}) if $0 \in \partial \Phi(u_0)$, where $\partial \Phi(u_0)$ denotes the generalized gradient of $\Phi$ in $u_0$, as defined in \cite{Chang}. It follows that this is equivalent to $\mathcal{F}'(u_0) \in \partial \mathcal{J}(u_0)$ and, since $\mathcal{J}$ is convex, this can be written as
\begin{equation}
\mathcal{J}(v) - \mathcal{J}(u_0) \geq \mathcal{F}'(u_0)(v-u_0), \quad \forall v \in BV(\mathbb{R}^N).
\label{BVsolution}
\end{equation}
Hence all $u_0 \in BV(\mathbb{R}^N)$ such that (\ref{BVsolution}) holds is going to be called a bounded variation solution of (\ref{P1}). 

\begin{proof}[Proof of Theorem \ref{theoremapplication}]
First of all let us prove that the restriction of $\Phi$ to the Banach space $BV_{rad}(\mathbb{R}^N)$ satisfies condition $i)$ of the Mountain Pass Theorem. But before that, just note that by $(f_2)$ and $(f_3)$ it follows that for all $\eta > 0$, there exists $A_\eta > 0$ such that
\begin{equation}
|F(s)| \leq \eta|s| + A_\eta|s|^p, \quad \forall s \in \mathbb{R}.
\label{fepsilon}
\end{equation}

Note that, by (\ref{fepsilon}) and the embeddings of $BV(\mathbb{R}^N)$, it follows that
\begin{eqnarray*}
\Phi(u)& = &\int_{\mathbb{R}^N} |Du| + \int_{\mathbb{R}^N} |u| dx - \int_{\mathbb{R}^N} F(u)dx\\
& \geq & \|u\| - \eta |u|_1 - A_\eta|u|_p^p\\
& \geq & \|u\| - \eta \|u\| - c_3\|u\|^p\\
& = & \|u\|\left(1 - \eta - c_3\|u\|^{p-1}\right)\\
& \geq & \alpha,
\end{eqnarray*}
for all $u \in BV_{rad}(\mathbb{R}^N)$, such that $\|u\| = \rho$, where $0 < \eta < 1$ is fixed, $\displaystyle 0 < \rho < \left(\frac{1-\eta}{c_3}\right)^\frac{1}{p-1}$ and $\displaystyle \alpha = \rho(1 - \eta - c_3\rho^{p-1})$.

Now let us prove that $\Phi$ satisfies the condition $ii)$ of Theorem \ref{mountainpass}.
First note that condition $(f_4)$ implies that there exists constants $d_1,d_2 > 0$ such that
\begin{equation}
F(s) \geq d_1|s|^\theta - d_2 ,\quad \forall s \in \mathbb{R}.
\label{Festimate}
\end{equation}

Let $u \in BV_{rad}(\mathbb{R}^N)$, with compact support, $u \neq 0$ and let $t > 0$. Then
$$
\Phi(tu) \leq t\|u\| - d_1t^\theta |u|_\theta^\theta + d_2|\mbox{supp}(u)| \to -\infty,
$$
as $t \to +\infty$, since $\theta > 1$.

Then, Theorem \ref{mountainpass} implies that, given a sequence $\epsilon_n \to 0$, there exists $(u_n)\subset BV_{rad}(\mathbb{R}^N)$ such that 
\begin{equation}
\lim_{n \to \infty}\Phi(u_n) = c
\label{mountainpasssequence1}
\end{equation}
and
\begin{equation}
\mathcal{J}(v) - \mathcal{J}(u_n) \geq \int_{\mathbb{R}^N}f(u_n)(v - u_n)dx - \epsilon_n\|v - u_n\|, \quad \forall v \in BV_{rad}(\mathbb{R}^N).
\label{mountainpasssequence}
\end{equation}

Let us prove that the sequence $(u_n)$ is bounded in $BV_{rad}(\mathbb{R}^N)$.
In (\ref{mountainpasssequence}) with $I_0 = \mathcal{J}$ and $I = \mathcal{F}$, let us take as test function $v = 2u_n$ and note that
$$
\|u_n\| \geq \int_{\mathbb{R}^N}f(u_n)u_ndx - \epsilon_n\|u_n\|,
$$
which implies that
\begin{equation}
(1 + \epsilon_n)\|u_n\| \geq \int_{\mathbb{R}^N}f(u_n)u_ndx.
\label{eqn1}
\end{equation}
Then, by $(f_4)$ and (\ref{eqn1}), note that
\begin{eqnarray*}
c + o_n(1) & \geq & \Phi(u_n)\\
& = & \|u_n\| + \int_{\mathbb{R}^N}\left(\frac{1}{\theta}f(u_n)u_n - F(u_n)\right)dx - \int_{\mathbb{R}^N}\frac{1}{\theta}f(u_n)u_ndx\\
& \geq & \|u_n\|\left(1- \frac{1}{\theta} - \frac{\epsilon_n}{\theta}\right)\\
& \geq & C\|u_n\|,
\end{eqnarray*}
for some $C > 0$ uniform in $n \in \mathbb{N}$. Then it follows that $(u_n)$ is bounded.

By the boundedness of $(u_n) \subset BV_{rad}(\mathbb{R}^N)$ and Theorem \ref{theorem1}, it follows that there exists $u \in BV_{rad}(\mathbb{R}^N)$ such that $u_n \to u$ in $L^r(\mathbb{R}^N)$ for all such $1 < r < 1^*$. Note that the limit function in $L^r(\mathbb{R}^N)$ is in fact a function of $BV_{rad}(\mathbb{R}^N)$ (see \cite{Giusti}[Theorem 1.9]).
By Lebesgue Dominated Convergence Theorem, $(f_3)$, the lower semicontinuity of $\mathcal{J}$ with respect to the $L^r(\mathbb{R}^N)$ convergence and the boundedness of $(u_n)$ in $BV_{rad}(\mathbb{R}^N)$, it follows calculating the $\limsup$ both sides of (\ref{mountainpasssequence}) that
\begin{equation}
\mathcal{J}(v) - \mathcal{J}(u) \geq \int_{\mathbb{R}^N} f(u)(v-u), \quad \forall v \in BV_{rad}(\mathbb{R}^N).
\label{solution}
\end{equation}

Note that up to now, we cannot guarantee that the limit function $u$ is in fact nontrivial. In order to do so, let us prove the following claim.

\noindent {\bf Claim.} There exist $\beta, r > 0$ and $(y_n) \subset \mathbb{R}^N$ such that 
\begin{equation}
\liminf_{n \to \infty} \int_{B_r(y_n)}|u_n|dx \geq \beta > 0.
\label{claim}
\end{equation}
In fact, on the contrary, up to a subsequence, we would have that for all $R > 0$,
$$
\lim_{n \to \infty} \sup_{y \in \mathbb{R}^N}\int_{B_r(y)}|u_n|dx = 0.
$$
Then, by our version of Lions' Lemma, Theorem \ref{lionslemma}, it would follows that 
\begin{equation}
u_n \to 0 \quad \mbox{in $L^s(\mathbb{R}^N)$ for all $s \in (1,1^*)$.}
\label{claim1}
\end{equation}

Now, note that by taking $v = 0$ in (\ref{mountainpasssequence}), by (\ref{Phiderivada}), it follows that $\Phi'(u_n)u_n \leq \epsilon_n\|u_n\| = o_n(1)$. Then
\begin{equation}
\Phi(u_n) - \Phi'(u_n)u_n = c + o_n(1)
\label{claim2}
\end{equation}
Note also that, since $\mathcal{J}'(u_n)u_n = \mathcal{J}(u_n)$, it follows by (\ref{fepsilon}) and the analogous of it to $f$, for all $\eta > 0$
\begin{eqnarray*}
\Phi(u_n) - \Phi'(u_n)u_n & = & \int_{\mathbb{R}^N}\left(f(u_n)u_n - F(u_n)\right)dx\\
& \leq & C\eta|u_n|_1 + C|u_n|_p^p\\
& \leq & C'\eta + o_n(1),
\end{eqnarray*}
which contradicts (\ref{claim2}) and prove the claim.

Now, by considering $\overline{u}_n(x):= u_n(x + y_n)$, since $\Phi$ is invariant by translation, we get a sequence $(\overline{u_n})$ which satisfies (\ref{mountainpasssequence1}) and (\ref{mountainpasssequence}) and moreover $\overline{u}_n \to \overline{u}$ in $L^r(\mathbb{R}^N)$, for all $r \in (1,1^*)$, where $\overline{u} \neq 0$, and $\overline{u}$ satisfies (\ref{solution}), i.e., is a bounded variation solution of (\ref{P1}). 

Now, in order to prove that $\overline{u}$ satisfies (\ref{solution}) for all $v \in BV(\mathbb{R}^N)$, we need a sort of Symmetric Criticality Principle of Palais (called here as SCPP), which is well known to hold for smooth functionals. Since here we are dealing with a functional which is given as the difference between a locally Lipschitz and a smooth functional, the classical version of SCPP cannot be used. This in fact is a great field of research and there are some papers dealing with the extension of this principle to functionals which are not smooth (see \cite{Otani} and \cite{Alexandru} for example). Our problem here is also worsened by the fact that $BV(\mathbb{R}^N)$ is not a reflexive space, property which becomes more simple the proof of the version of SCPP to a non-smooth setting. Fortunately, in \cite{Squassina}, Squassina succeed in proving a version of SCPP in a situation which comprises exactly our situation. Hence, by \cite{Squassina}[Theorem 4], it follows that $\overline{u}$ satisfies (\ref{solution}) for all $v \in BV(\mathbb{R}^N)$ and then is a nontrivial bounded variation solution of (\ref{P1}).

Now, what is left to justify is just that the solution $\overline{u}$ in fact is a ground-state solution, i.e., that $\overline{u}$ has the lowest energy level among all nontrivial bounded variation solutions. In order to prove it, we have to recall \cite{FigueiredoPimenta}, where is proved that we can define the Nehari set associated to $\Phi$, given by
$$
\mathcal{N} = \left\{u \in BV(\mathbb{R}^N)\backslash\{0\}; \, \int_{\mathbb{R}^N}|Du| + \int_{\mathbb{R}^N}|u|dx = \int_{\mathbb{R}^N}f(u)udx   \right\}.
$$
It is proven in \cite{FigueiredoPimenta} that $\mathcal{N}$ is a set which contains all nontrivial bounded variation solutions of (\ref{P1}). Then, if we manage to prove that the solution $\overline{u}$ is such that $\Phi(\overline{u}) = \inf_{\mathcal{N}}\Phi$, then $\overline{u}$ would be a ground-state solution of (\ref{P1}).

By using the same kind of arguments that Rabinowitz in \cite{Rabinowitz}, in the light of $(f_1) - (f_5)$, one can easily see that $\mathcal{N}$ is radially homeomorphic to the unit sphere in $BV(\mathbb{R}^N)$ and also that the minimax level $c$ satisfies
$$
c = \inf_{v \in BV(\mathbb{R}^N)\backslash\{0\}}\max_{t \geq 0} \Phi(tv) = \inf_{v \in \mathcal{N}}\Phi(v).
$$
Since the solution $\overline{u}$ is such that $\Phi(\overline{u}) = c$, it follows that $u$ is a ground-state bounded variation solution.

\end{proof}

\end{document}